\documentclass[11pt,letterpaper]{amsart}

\usepackage{amssymb}
\usepackage{enumerate}

\newtheorem{theorem}{Theorem}
\newtheorem{lemma}[theorem]{Lemma}

\newcommand{\abs}[1]{\lvert#1\rvert}

\allowdisplaybreaks[3]

\begin{document}

\title{Barker sequences of odd length}

\author{Kai-Uwe Schmidt}
\address{Faculty of Mathematics, Otto-von-Guericke University, Universit\"atsplatz~2, 39106 Magdeburg, Germany}
\email[K.-U. Schmidt]{kaiuwe.schmidt@ovgu.de}

\author{J\"urgen Willms}
\address{Institut f\"ur Computer Science, Vision and Computational Intelligence, Fachhochschule S\"udwestfalen, 59872 Meschede, Germany}
\email[J. Willms]{willms.juergen@fh-swf.de}

\date{23 January 2015}


\begin{abstract}
A Barker sequence is a binary sequence for which all nontrivial aperiodic autocorrelations are at most $1$ in magnitude. An old conjecture due to Turyn asserts that there is no Barker sequence of length greater than $13$. In 1961, Turyn and Storer gave an elementary, though somewhat complicated, proof that this conjecture holds for odd lengths. We give a new and simpler proof of this result.
\end{abstract}

\maketitle

%%%%%%%%%%%%%%%%%%%%%%%%%%%%%%%%%%%%%%%%%%%%%%%%%%%%%%%%%%%%%%%%

\section{Introduction}

Consider a binary sequence $A$ of length $n>1$, namely an element of $\{-1,1\}^n$. We write $A(k)$ for the $k$-th entry in $A$. Define the \emph{aperiodic autocorrelation} at shift~$u$ (where $0\le u<n$) of~$A$ to be
\[
C(u)=\sum_{k=1}^{n-u}A(k)A(k+u).
\] 
Notice that $C(0)=n$. All other values $C(u)$ are called the \emph{nontrivial} aperiodic autocorrelations. There is sustained interest in binary sequences for which all of the nontrivial aperiodic autocorrelations are small (see~\cite{Jed2008} for a good survey). It is known~\cite{Sch2012} that, for each $n>1$, there exists a binary sequence of length $n$ such that all nontrivial aperiodic autocorrelations are at most $\sqrt{2n\log(2n)}$ in magnitude.
\par
On the other hand, it is not known whether there exist infinitely many Barker sequences, namely binary sequences with the ideal property that the nontrivial aperiodic autocorrelations are at most $1$ in magnitude. Notice that for fixed $a,b\in\{0,1\}$, the transformation $A(k)\mapsto A(k)(-1)^{a+bk}$ preserves the Barker property. We can therefore assume  without loss of generality that a Barker sequence $A$ satisfies $A(1)=A(2)=1$. The only known Barker sequences with this property are (writing $+$ for $1$ and~$-$~for~$-1$)
\begin{alignat*}{3}
A_3&=[++-],    &\quad A_2&=[+\,+],\\[.3ex]
A_5&=[+++-+],  &\quad A_4&=[+++\,-],\\[.3ex]
A_7&=[+++--+-],&\quad A'_4&=[++-\,+],\\[.3ex]
A_{11}&=[+++---+--+-],\\[.3ex]
A_{13}&=[+++++--++-+-+].
\end{alignat*}
Indeed, it has been conjectured since at least 1960~\cite{Tur1960} that there is no Barker sequence of length greater than $13$. This conjecture is known to be true for sequences of odd length, as proven by Turyn and Storer~\cite{TurSto1961}.
\begin{theorem}[{Turyn and Storer~\cite{TurSto1961}}]
\label{thm:main}
If there exists a Barker sequence of odd length $n$, then $n\in\{3,5,7,11,13\}$.
\end{theorem}
\par
Fairly deep methods have been devised to attack the case that the length is even, including the character-theoretic approach by Turyn~\cite{Tur1965} and the field-descent method by B. Schmidt~\cite{Sch1999}, but the problem remains open. We refer to~\cite{Jed2008} for a brief survey and to~\cite{BorMos2013} for the latest results on this problem.
\par
The proof of Theorem~\ref{thm:main} due to Turyn and Storer~\cite{TurSto1961} is elementary, but involves an arduous inductive argument.\footnote{We note that~\cite{Wil2014} gives counterexamples to~\cite[Theorem 1 (iv)]{TurSto1961}. One can show however that, in~\cite[Theorem 1]{TurSto1961}, the statements (ii) and (iii) imply~(iv) with the corrected range $k\le t/p-1/2$, which is consistent with~\cite{Wil2014} and sufficient for the induction in the proof.} Borwein and Erd{\'e}lyi~\cite{BorErd2013} gave a proof using a different induction, but its overall structure is similar to that of Turyn and Storer~\cite{TurSto1961}. In this paper, we offer a simpler proof of Theorem~\ref{thm:main}.
\par
We briefly explain how our proof differs from the previous ones. For a putative Barker sequence $A$ of length $n>3$ assume that $A(1)=A(2)=1$ and let $p+1$ be the position of the first occurrence of $-1$ in $A$. It is not hard to show that $p\ge 3$. The crucial (and lengthy) step in the proofs of~\cite{TurSto1961} and~\cite{BorErd2013} is to establish that $A$ has the following block structure
\[
A(jp+1)=A(jp+2)=\dots=A(jp+r)
\]
for all $j$ and $r$ satisfying $1\le jp+r\le n-p-2$ and $1\le r\le p$. Once this is established, it is easy to conclude that $A$ cannot have many such blocks and must therefore be short.
\par
In contrast, we do not establish such a block structure explicitly. We consider the \emph{runs} of $A$, which are subsequences of maximal length consisting of equal entries (see~\cite{Wil2013} for connections between runs and autocorrelations). We assume that $A$ starts with $e-1$ runs whose lengths are divisible by $p$ and that the length of the $e$-th run is not divisible by $p$. Let~$q$ be the sum of the lengths of the first~$e$ runs. It is not hard to show that $n\ge 2q-3$. The key result is that, if $n>p+q+1$, then
\[
\abs{C(n-p-q+1)-C(n-p-q-1)}\ge 4,
\]
which contradicts the defining property of a Barker sequence. Therefore,~we have $2q-3\le n\le p+q+1$, from which we can easily deduce Theorem~\ref{thm:main}.
\par
We shall make use of the following results due to Turyn and Storer~\cite{TurSto1961}. In order to make this note self-contained, we include their short proofs. 
\begin{lemma}[{Turyn and Storer~\cite{TurSto1961}}]
\label{lem:TS}
Suppose that $A$ is a Barker sequence of odd length~$n$. Then the following statements hold:
\begin{enumerate}[(i)]
\item\label{itm:ss} $A(k)A(n-k+1)=(-1)^{(n+1)/2+k}$ for each $k$ satisfying $1\le k\le n$.
\item\label{itm:doubling} $A(k)A(k+1)=A(2k)A(2k+1)$ for each $k$ satisfying $1\le k\le \frac{n-3}{2}$.
\end{enumerate}
\end{lemma}
\begin{proof}
First note that, if $u$ is odd, then $C(u)$ is a sum of an even number of $1$ or $-1$, so $C(u)=0$. The next step is to observe that, for $0<u<n$,
\begin{equation}
C(u)+C(n-u)=\sum_{k=1}^nA(k)A(k+u),   \label{eqn:periodic}
\end{equation}
where the second index is reduced modulo $n$ if necessary. Use $xy\equiv x-y+1\pmod 4$ for $x,y\in\{-1,1\}$ to conclude that~\eqref{eqn:periodic} is congruent to $n$ modulo $4$. Therefore, since exactly one of $u$ and $n-u$ is odd, we find that
\begin{equation}
C(u)\equiv \begin{cases}
0\pmod 4 & \text{for odd $u$}\\
n\pmod 4 & \text{for even $u$}.
\end{cases}   \label{eqn:Cu}
\end{equation}
Now count the number of $1$ and $-1$ in the sum $C(u)$ to obtain, for $0\le u<n$,
\[
\prod_{k=1}^{n-u}A(k)A(k+u)=(-1)^{(n-u-C(u))/2}.
\]
Multiply two successive equations of this form and use~\eqref{eqn:Cu} to prove (\ref{itm:ss}).
\par
To prove (\ref{itm:doubling}), use (\ref{itm:ss}) to obtain, for $1\le u\le \frac{n-1}{2}$,
\begin{align*}
C(n-2u+1)&=\sum_{k=1}^{2u-1}A(k)A(2u-k)(-1)^{\frac{n+1}{2}+k}\\
&=A(u)^2(-1)^{\frac{n+1}{2}+u}+2\sum_{k=1}^{u-1}A(k)A(2u-k)(-1)^{\frac{n+1}{2}+k}.
\end{align*}
By~\eqref{eqn:Cu}, the left-hand side equals $(-1)^{(n-1)/2}$, so that
\[
-\frac{1+(-1)^u}{2}=\sum_{k=1}^{u-1}A(k)A(2u-k)(-1)^k.
\]
Count the number of $1$ and $-1$ in the sum to find that
\[
\prod_{k=1}^{u-1}A(k)A(2u-k)=1
\] 
or equivalently
\[
\prod_{k=1}^{2u-1}A(k)=A(u).
\] 
Multiplying two successive equations of this form proves (\ref{itm:doubling}).
\end{proof}

%%%%%%%%%%%%%%%%%%%%%%%%%%%%%%%%%%%%%%%%%%%%%%%%%%%%%%%%%%%%

\section{Proof of Theorem~\ref{thm:main}}

Suppose that $A$ is a Barker sequence of odd length $n=2m-1$. Since $C(1)$ is a sum of an even number of $1$ or $-1$, we have $C(1)=0$. This implies that $A$ has exactly $m$ runs. Accordingly, we associate with $A$ the unique numbers $s_0,s_1,\dots,s_m$ satisfying
\[
0=s_0<s_1<\dots<s_{m-1}<s_m=n
\]
and
\begin{equation}
A(s_j+1)=A(s_j+2)=\cdots=A(s_{j+1})=(-1)^jA(1)   \label{eqn:def_s}
\end{equation}
for all $j\in\{0,1,\dots,m-1\}$. Note that $s_1>1$ implies $s_m=s_{m-1}+1$ by Lemma~\ref{lem:TS}~(\ref{itm:ss}), so that~$s_1$ cannot divide all of the numbers $s_1,\dots,s_m$. Accordingly, for $s_1>1$, we define $e$ to be the smallest $j$ such that $s_1\nmid s_j$.
\par
We shall need two lemmas.
\begin{lemma}
\label{lem:ub_n}
Suppose that $A$ is a Barker sequence of odd length $n>5$ and that $s_1>1$. Then  $s_1$ and $s_e$ are odd and $n\ge 2s_e-3$.
\end{lemma}
\begin{proof}
From Lemma~\ref{lem:TS}~(\ref{itm:ss}), we find that $A$ must end with $s_1$ alternating entries, which implies that $n\ge 2s_1-1$. Since $n>5$, we then conclude that, if $s_1$ were even, then Lemma~\ref{lem:TS}~(\ref{itm:doubling}) gives
\[
A(s_1/2)A(s_1/2+1)=A(s_1)A(s_1+1),
\]
which contradicts~\eqref{eqn:def_s}. Hence $s_1$ is odd and so $s_1\ge 3$.
\par
For the lower bound for $n$, recall that $A$ starts with $e$ blocks of equal entries, where the first $e-1$ blocks have length at least $s_1\ge 3$. By Lemma~\ref{lem:TS}~(\ref{itm:ss}), $A$ must end with $e$ blocks of alternating elements whose lengths match those of the corresponding initial locks. The overlap between the initial blocks and the corresponding final blocks can be at most $3$. Hence $n\ge 2s_e-3$.
\par
It remains to show that $s_e$ is odd. We know that $s_e\le (n+3)/2$. Hence, if $s_e$ were even, then since $n>5$, we find from Lemma~\ref{lem:TS}~(\ref{itm:doubling}) that
\[
A(s_e/2)A(s_e/2+1)=A(s_e)A(s_e+1),
\]
which again contradicts~\eqref{eqn:def_s} since $s_1$ does not divide $s_e$.
\end{proof}
\par
Our key result is the following lemma.
\begin{lemma}
\label{lem:lb_n}
Suppose that $A$ is a Barker sequence of odd length $n$ and that $s_1>1$. Then $n\le s_1+s_e+1$.
\end{lemma}
\begin{proof}
Since $e>1$ and $s_1>1$, the lemma holds for $n\le 5$, so assume that $n>5$. From Lemma~\ref{lem:ub_n} we know that $s_1$ and $s_e$ are odd. Write $v=s_1+s_e$, so that $v$ is even, and suppose for a contradiction that $n\ge v+3$.
\par
In what follows, we make repeated use of~\eqref{eqn:def_s} without explicit reference. Without loss of generality, we can assume that $A(1)=1$. Let $u$ be an even integer satisfying $s_e+1\le u \le n-1$. From Lemma~\ref{lem:TS}~(\ref{itm:ss}) we find that
\[
C(n-u+1)=\sum_{k=1}^{u-1}A(k)A(u-k)(-1)^{\frac{n+1}{2}+k},
\]
which we can rewrite as
\[
(-1)^{\frac{n+1}{2}}\,C(n-u+1)=\sum_{j=0}^{e-1}(-1)^j\!\!\!\sum_{k=s_j+1}^{s_{j+1}}\!\!\!A(u-k)(-1)^k+\!\!\!\sum_{k=s_e+1}^{u-1}\!\!\!A(k)A(u-k)(-1)^k.
\]
Since $s_e+2\le v\le n-3$ by assumption, we can apply this identity with $u=v$ and $u=v+2$ to obtain
\begin{equation}
(-1)^{\frac{n+1}{2}}(C(n-v+1)-C(n-v-1))=\sum_{j=0}^{e-1}(-1)^jS_j+R-A(v)+A(v+1),   \label{eqn:diff_C_sums}
\end{equation}
where
\[
S_j=\sum_{k=s_j+1}^{s_{j+1}}(-1)^k(A(v-k)-A(v-k+2))
\]
for $0\le j\le e-1$ and
\[
R=\sum_{k=s_e+1}^{v-1}(-1)^kA(k)(A(v-k)-A(v-k+2)).
\]
Since $v\ge s_e+2$, the sum $R$ is nonempty. However only the first summand in~$R$ is nonzero. Hence, since $s_e$ is odd,
\[
R=A(s_e+1)(A(s_1-1)-A(s_1+1))=2(-1)^e.
\]
The sum $S_j$ is telescoping, so we can rewrite $S_j$ as
\[
S_j=(-1)^{s_{j+1}}(A(v-s_{j+1})-A(v-s_{j+1}+1))-(-1)^{s_j}(A(v-s_j)-A(v-s_j+1)),
\]
from which we find that
\[
S_j=0\quad\text{for $1<j<e-1$}
\]
since, by definition, $v-s_j$ and $v-s_{j+1}$ are not divisible by $s_1$ for $1<j<e-1$. Moreover, we obtain
\[
S_0=2(-1)^e-A(v)+A(v+1)
\]
and
\[
\text{$S_1=-2(-1)^e$ and $S_{e-1}=-2$ for $e>2$}.
\]
For $e=2$, we have $S_1=-4$. Substitute everything into~\eqref{eqn:diff_C_sums} to give
\[
(-1)^{\frac{n+1}{2}}(C(n-v+1)-C(n-v-1)=8(-1)^e-2A(v)+2A(v+1),
\]
which contradicts the Barker property of $A$. Therefore $n\le v+1$.
\end{proof}
We now complete our proof of the theorem. We know that $A_3$ and $A_5$ are Barker sequences of length $3$ and $5$, respectively, so assume that $n>5$. As mentioned earlier, we can assume without loss of generality that $A(1)=A(2)=1$, so that $s_1>1$. From Lemmas~\ref{lem:ub_n} and~\ref{lem:lb_n} we then find that
\begin{equation}
2s_e-3\le n\le s_1+s_e+1,   \label{eqn:bounds_n}
\end{equation}
which implies $s_e\le s_1+4$. From Lemma~\ref{lem:ub_n} we know that $s_1$ and $s_e$ are odd and that $s_1\ge 3$. Since $s_1\ge 3$, we have $e\in\{2,3\}$ and, since $s_e-s_1$ is even, there are only the following three cases to consider.
\par
{\itshape Case 1: $e=3$.} This case forces $(s_1,s_2,s_3)=(3,6,7)$, so $n=11$ by~\eqref{eqn:bounds_n} and the corresponding sequence is $A_{11}$.
\par
{\itshape Case 2: $e=2$ and $s_2=s_1+2$.} Here~\eqref{eqn:bounds_n} implies that $n$ equals either $2s_1+1$ or $2s_1+3$. Hence, we find from Lemma~\ref{lem:TS}~(\ref{itm:doubling}) that
\[
1=A(s_1-1)A(s_1)=A(2s_1-2)A(2s_1-1).
\]
If $n=2s_1+1$, then $A(n-2)=A(n-3)$, which forces $s_1=3$ by Lemma~\ref{lem:TS}~(\ref{itm:ss}). Therefore, in this case we have $n=7$ and $(s_1,s_2)=(3,5)$ and the corresponding sequence is $A_7$. If $n=2s_1+3$, then $A(n-4)=A(n-5)$, which implies that $s_1=3$ or $5$. In the first case we obtain $n=9$ and $(s_1,s_2)=(3,5)$. But then Lemma~\ref{lem:TS}~(\ref{itm:ss}) and~(\ref{itm:doubling}) imply $A(6)=A(7)=1$ and $A(6)A(7)=-1$, respectively, a contradiction. In the second case we obtain $n=13$ and $(s_1,s_2)=(5,7)$ and the corresponding sequence is $A_{13}$.
\par
{\itshape Case 3: $e=2$ and $s_2=s_1+4$.} In this case we have $n=2s_1+5$ by~\eqref{eqn:bounds_n}. Since, by Lemma~\ref{lem:TS}~(\ref{itm:ss}), $A$ must end with a block of $s_1$ alternating elements preceded by a block of four alternating elements, we have $n\ge 2s_2-1$. Hence $n\ge 2s_1+7$, a contradiction.

%%%%%%%%%%%%%%%%%%%%%%%%%%%%%%%%%%%%%%%%%%%%%%%%%%%%%%%%%%%%%%%%

% \bibliographystyle{plain}
% \bibliography{references}

\end{document}